\documentclass[a4paper,11pt]{article}
\usepackage{amssymb,amsmath,amsthm,latexsym,url}
\newtheorem{lem}{Lemma}[section]
\newtheorem{thm}[lem]{Theorem}
\newtheorem{cor}[lem]{Corollary}
\newtheorem{pro}[lem]{Proposition}
\newtheorem{exa}[lem]{Example}
\title{\normalsize\bf Groups generated by a finite Engel set}
\author{\small\textsc{Alireza Abdollahi}\footnote{The first author's research was in part supported by a grant from
IPM (No. 90050219) and by the Center of Excellence for Mathematics, University of Isfahan.}\\
\small{Department of Mathematics, University of Isfahan}\\
\small{81746-73441, Isfahan, Iran;}\\
\small{and School of Mathematics, Institute for Research in Fundamental Sciences (IPM)}\\
\small{P.O.Box: 19395-5746, Tehran, Iran}\\
\small{Email: {\tt a.abdollahi@math.ui.ac.ir}}\\[10pt]
\small\textsc{Rolf Brandl}\\
\small{Mathematisches Institut, Am Hubland 12}\\
\small{97074 W\"urzburg, Germany}\\[10pt]
\small\textsc{Antonio Tortora}\footnote{The third author is grateful to Professor A. Rhemtulla for interesting discussions and
useful suggestions. He also wishes to thank the Department of
Mathematical and Statistical Sciences at the University of Alberta for its fine hospitality while part of this work was being carried out.}\\
\small{Dipartimento di Matematica, Universit\`a di Salerno}\\
\small{Via Ponte don Melillo, 84084 - Fisciano (SA), Italy}\\
\small{E-mail: {\tt antortora@unisa.it}}}
\date{}

\begin{document}
\maketitle


\begin{abstract}
\noindent A subset $S$ of a group $G$ is called an Engel set if, for all
$x,y\in S$, there is a non-negative integer $n=n(x,y)$ such that
$[x,\,_n y]=1$. In this paper we are interested in finding conditions for a group
generated by a finite Engel set to be nilpotent. In particular, we focus our investigation on groups generated by an Engel set of size two.\\

\end{abstract}

\section{Introduction}

\noindent A subset $S$ of a group $G$ is called an Engel set if, for all
$x,y\in S$, there is a non-negative integer $n=n(x,y)$ such that
$[x,\,_n y]=1$. It is known that, for a group $G$ satisfying Max-$ab$, a normal subset
$S\subseteq G$ is an Engel set if and only if it is contained in the Fitting subgroup of $G$
(see \cite{Ro}, Theorem 7.23; see also \cite{Ab}) and so, in this case,  $\langle S\rangle$ is nilpotent
whenever $S$ is finite. However, a group generated by a finite Engel set is not necessarily
nilpotent: Golod's examples show that there exist infinite non-nilpotent
groups generated by an Engel set with three or more elements (see \cite{Go}).
Furthermore, if $S$ is an Engel set of size three, then an easier example of a non-nilpotent
group generated by $S$ is the wreath product of the alternating
group of degree 5 with the cyclic group of order 3: it has a {\em presentation
of type $(r,s,t)$} (see \cite{ET}), i.e. $S=\{a,b,c\}$ where $\langle
a,b\rangle$ is nilpotent of class $r$, $\langle a,c\rangle$ is
nilpotent of class $s$ and $\langle b,c\rangle$ is nilpotent of
class $t$. All these groups are not soluble, but the nilpotency
does not hold even in the soluble case. In \cite{ET} it was shown that
every group with a presentation of type (1,2,2) is soluble of length
at most 3 and that there are non-nilpotent groups of this type.

In this paper, we first get that any nilpotent-by-abelian group
generated by a finite Engel set is nilpotent and then we focus on groups
generated by an Engel set of size two. In particular, we prove that such a group
is nilpotent whenever it is abelian-by-(nilpotent of class 2).
This is the best possible result in the soluble case. In fact,
we construct by {\sf GAP} (see \cite{GAP4})
a non-nilpotent counterexample which is abelian-by-(nilpotent of class 3).
On the other hand, some of the counterexamples in \cite{ET},
mentioned above, are abelian-by-(nilpotent of class 2)
and generated by an Engel set of size three.

\section{Groups that are Nilpotent-by-Abelian}

\noindent We start with a result that is certainly already known. It generalizes, for metabelian groups, two basic properties of commutators.

\begin{lem}\label{comm} Let $G$ be a metabelian group and
$x,y,z$ be elements of~$G$. For all positive integers $n$, we have:
\begin{itemize}
\item[$(i)$] $[x^{-1},\,_n y]=[x,\,_n y]^{-x^{-1}}$;
\item[$(ii)$] $[xy,\,_n z]=[x,\,_n z][x,\,_n z,y][y,\,_n z]$.
\end{itemize}
\end{lem}

\begin{proof}
Since $G$ is metabelian, every $g$ in $G$ induces on $G'$ an endomorphism $-1+g$ that maps
$u$ to $u^{-1}u^g$, and any two of these commute. We thus have:
$$[x^{-1},\,_n y]=([x,y]^{-x^{-1}})^{(-1+y)^{n-1}}=
[x,y]^{-(-1+y)^{n-1}x^{-1}}=[x,\,_n y]^{-x^{-1}}\,.$$
The proof of $(ii)$ is similar.
\end{proof}

As a consequence of Lemma \ref{comm}, we get:

\begin{lem}\label{metabelian} If $G$ is a metabelian group generated
by an Engel set $S$, then any $x\in S$ is a left Engel element. In
particular, $G$ is locally nilpotent.
\end{lem}

\begin{proof} Take a finite subset of $S$, say $T=\{x_1,\ldots,x_r\}$,
and suppose $[x_i,\,_n x_j]=1$ for all $1\leq i,j\leq r$. By the previous lemma,
every $x_i$ is a left $n$-Engel element in $G$.
Then $(-1+x_i)^n=0$. It follows that any product in the endomorphisms $-1+x_i$ of weight $(n-1)r+1$ is trivial. Hence
$\langle T\rangle$ is nilpotent of class at most $(n-1)r+2$. This proves that $G$ is locally nilpotent.
\end{proof}

For a finite Engel set, we then obtain the following:

\begin{thm} Let $G$ be a nilpotent-by-abelian group generated by a
finite Engel set. Then $G$ is nilpotent.
\end{thm}

\begin{proof} If $N$ is a normal nilpotent subgroup of $G$ such that $G/N$ is
abelian, then $G/N'$ is nilpotent by Lemma \ref{metabelian} and so
$G$ is nilpotent by a well-known result of P.\,Hall.
\end{proof}

\section{Engel sets of size two}

\noindent Let $G=\langle x,y\rangle$ be a group and assume that
$\{x,y\}$ is an Engel set. Then $[x,\,_n y]=1$ and $[y,\,_m x]=1$
for some positive integers $n,m$. We also say
 that the elements $x$ and $y$ are {\em mutually Engel} and, whenever $n\geq m$,
that they are {\em mutually n-Engel}. If $n=m=2$, then $G$ is obviously
nilpotent of class at most 2 and the nilpotency still holds for $n=2$ and $m=3$.

\begin{pro}
Let $G=\langle x,y\rangle$ be an arbitrary group such that
$[x,y,y]=1$ and $[y,x,x,x]=1$. Then $G$ is nilpotent of class at most
$3$.
\end{pro}

\begin{proof}
By the Hall-Witt identity we have
$$[[y,x],x^{-1},y]^x[x,y^{-1},[y,x]]^{y}[y,[y,x]^{-1},x]^{[y,x]}=1,$$ from which
it follows $$[y,x,x^{-1},y]=1$$
since $[x,y^{-1}]=[x,y]^{-1}$ and
$[y,[y,x]^{-1}]=[x,y,y]^{-1}=1$. Then $[y,x,x,y]=1$ and hence
$[y,x,x]$ $\in Z(G)$. Now $[x,y,y]=[y,x,x]=1$ modulo $Z(G)$, so
$G/Z(G)$ is nilpotent of class $\leq 2$ and $G$ is nilpotent of
class $\leq 3$.
\end{proof}

However, as we will see in the next section, this is not true in general, even in the soluble case. We are therefore led to consider extra
conditions for a group generated by an Engel set of size two to be nilpotent. In the sequel, we will turn our attention to groups which are abelian-by-(nilpotent of class 2).\\

Let $G$ be any abelian-by-(nilpotent of class 2) group
generated by two mutually Engel elements $x$ and $y$. By assumption
$[x,\,_n y]=1$ and $[y,\,_n x]=1$ for some $n$. Suppose, by way of contradiction,
that $G$ is not nilpotent. Then $G$ has a non-nilpotent finite image by Theorem 10.51 of \cite{Ro} and so
we may assume that $G$ is finite.

Using induction on the order of
the group, we may assume that $G$ is a minimal counterexample. It follows
that $G$ contains a unique minimal normal subgroup $A$ such that $G/A$
is nilpotent. As $G$ is not nilpotent there is a maximal subgroup $H$ that is not normal. On the other hand $G/A$ is nilpotent,
therefore $A\nleq H$ (otherwise $H/A\lhd G/A$ implies that $H\lhd G$).
Thus $G=AH$. The group $A\cap H$ is normal in $G$ and $A\cap H<A$. The minimality of $A$ then forces $A\cap H=1$.

Clearly, $A$ is an elementary abelian $p$-group for some prime $p$ and $H$ is
nilpotent. Let $P$ be the Sylow $p$-subgroup of $H$. Then $AP/A\lhd G/A$ and so $AP$ is the Sylow $p$-subgroup of $G$.
Since $AP$ is nilpotent, we have that $[A,AP]<A$ and by the minimality of $A$, the normal subgroup
$[A,AP]$ must be trivial. Thus $[A,P]=1$ and $P^G=P^{AH}=P^H=P$, that is $P\lhd G$. But $A\nleq P$, hence $P=1$ and
$H$ is a Hall $p'$-subgroup of $G$.

\begin{lem}\label{1} Every nontrivial element of $Z(H)$ acts
fixed point freely on $A$ by conjugation.
\end{lem}

\begin{proof} For all $z\in Z(H)$ and $h\in H$, $C_A(z)^h=C_A(z)$ and thus $C_A(z)\lhd G$.
As $\langle z\rangle$ cannot be normal in $G$, we get $C_A(z)=1$ by minimality of $A$.
\end{proof}

The next lemma shows that $H$ is nilpotent of class 2 and that
we can restrict our attention to $n=3$.

\begin{lem}\label{2} Let $G=AH=\langle x,y\rangle$
be a minimal counterexample that is abelian-by-(nilpotent of class
$2$). Then $A=\gamma_{3}(G),[x,y,y,y]=1$ and $[y,x,x,x]=1$.
\end{lem}

\begin{proof}
Of course, $A\subseteq\gamma_{3}(G)$ by minimality of $A$.
Let $q\neq p$ be a prime. Then any $q$-subgroup of $\gamma_{3}(G)$
is necessarily trivial. But $G/A$ is a $p'$-group, therefore $A=\gamma_{3}(G)$ and $H$ is nilpotent
of class 2.

Assuming now $[x,\,_{n-1}y]\neq1$, we will prove that $n\leq 3$. Let $y=ah$ where
$a\in A,h\in H$, and suppose $n>3$. We have $[x,y,y]\in A$ and $n-2$ $\geq 2$,
so that $[x,\,_{n-2} y]$ and $[x,\,_{n-2} y,y]$ lie in $A$. It follows that
$$[x,\,_{n-2}y,y^p]=[x,\,_{n-2} y,y]^p=1.$$
Notice that $y^p=a_1h^p$ with $a_1\in A$ and $h=h^{\alpha p}$ for some integer $\alpha$.
Thus
$$1=[x,\,_{n-2} y,y^p]=[x,\,_{n-2} y,a_1h^p]=[x,\,_{n-2} y,h^p]$$
and
$$1=[x,\,_{n-2} y,h^{\alpha p}]=[x,\,_{n-2} y,h].$$
But then $$1=[x,\,_{n-2} y,ah]=[x,\,_{n-2} y,y],$$ that is a contradiction.
\end{proof}

We need one more preliminary lemma before proving our main result.

\begin{lem}\label{3} Let $x=ah,y=bk$ where
$a,b\in A$ and $h,k\in H$. If $[x,y]=[h,k]$, then
$$[a,k^{-1}]=[b,h^{-1}],\quad [a,h]=1\quad and \quad [b,k]=1,$$ with
$a\neq 1$ and $b\neq 1$.
\end{lem}

\begin{proof} We have $$[h,k]=[x,y]=[ah,bk]=[a,k]^h[h,k][h,b]^k.$$
This implies $[a,k]^h[h,b]^{h^{-1}kh}=1$ and then
$[a,k]^{k^{-1}}=[b,h]^{h^{-1}}$, or equivalently
$[a,k^{-1}]=[b,h^{-1}]$.

As $G\neq H$ we must have that one of $a,b$ is nontrivial.
Without loss of generality, we may assume $a\neq 1$.
Clearly, $[y,x,x]\in A$ and $1\neq[y,x]\in Z(H)$.
Then $1=[y,x,x,x]=[y,x,x,h]$ and $$[x,h]^{[y,x]}=[x^{[y,x]},h]=[[y,x,x]^{-1}x,h]=[x,h].$$
Thus $1=[x,h,[y,x]]=[[a,h]^h,[y,x]]=[a,h,[y,x]]^h$, so $[a,h]$ is fixed by $[y,x]$.
By Lemma \ref{1} it follows that $[a,h]=1$. As a consequence $b\neq 1$, otherwise $[a,k]=1$ and $[a,[h,k]]=1$.
Arguing as for $a$, we then conclude that $[b,k]=1$.
\end{proof}

\begin{thm}\label{Ab-by-N2} Let $G$ be any abelian-by-(nilpotent of class $2$) group
generated by two mutually Engel elements $x$ and $y$. Then $G$ is
nilpotent.
\end{thm}

\begin{proof} Put $x=ah,y=bk$ where $a,b\in A$ and $h,k\in
H$. Then $[x,y]=[h,k]c$ with $[h,k]\in Z(H)$ and for some $c\in A$.
By Lemma \ref{2} we know that
$$[x,y,y],[y,x,x]\in A \quad{\rm and}\quad [x,y,y,y]=[y,x,x,x]=1.$$
This gives  $$[x,y,y^p]=1 \qquad{\rm and}\qquad [x,y,x^p]=1.$$
If $\langle x^p,y^p\rangle\cap A\neq 1$, the commutator $[x,y]$ commutes with a
nontrivial element of $A$. Thus $[h,k]=1$ by Lemma
\ref{1}, and $[x,y]\in A$. Indeed $G'\leq A$
and $G$ is nilpotent by Lemma \ref{metabelian}. Therefore
$A\cap\langle x^p,y^p\rangle=1$ and we may assume $H=\langle
x^p,y^p\rangle$, since $\langle h,k\rangle\simeq\langle h,k\rangle
A/A=\langle x^p,y^p\rangle A/A\simeq \langle x^p,y^p\rangle$.
It follows that $c$ must be trivial. Then $1\neq[x,y]=[h,k]$ and, by Lemma
\ref{3}, we have
$$[a,k^{-1}]=[b,h^{-1}] \qquad{\rm and}\qquad [a,h]=1,$$
with $a\neq 1$.

Now, the Hall-Witt identity
$$[a,k^{-1},h]^k[k,h^{-1},a]^h[h,a^{-1},k]^a=1$$
implies
$$[a,k^{-1},h]^{k}=[k,h^{-1},a]^{-h}.$$
But $[k,h^{-1},a]$ commutes
with $h$, so $[[a,k^{-1}],h]=[[b,h^{-1}],h]$ commutes with
$h^{k^{-1}}$. Then $[b,h,h]^{h^{-1}}=[b,h^{-1},h]^{-1}$
commutes with $h^{k^{-1}}$, in particular $[b,h,h]$ commutes with
$h^{k^{-1}h}=h^{k^{-1}}$. Hence $[b,h,h]\in C_A(h^{k^{-1}})$.

Let $B=C_A(h^{k^{-1}})$ and $K=\langle h,h^{k^{-1}}\rangle A$. Then
$B\lhd K$ because $[h^{-1},k]\in Z(H)$. If $q$ is
the order of $h$, we also have $B=[b,h^q]B=[b,h]^qB$. However, the order of
$[b,h]$ is coprime with $q$, thus $[b,h]\in B$ and $[a,k^{-1}]=[b,h^{-1}]\in B$. So
$[a,k^{-1},h^{k^{-1}}]=1$ and $[k,a,h]=1$. Finally, from
$$[a,k,h]^{k^{-1}}[k^{-1},h^{-1},a]^h[h,a^{-1},k^{-1}]^a=1,$$
it follows $[k,h,a]=1$ which contradicts Lemma \ref{1}.
\end{proof}

When $x$ and $y$ are mutually 3-Engel elements, we get thanks to {\sf GAP} that the group $G$ in Theorem \ref{Ab-by-N2}
is nilpotent of class at most $8$. In fact, using the ANU {\sc Nilpotent Quotient} package of W. Nickel
(see \cite{Nickel}), we can construct the largest nilpotent quotient of $G$ which is isomorphic to $G$.

Also notice that the theorem above can be extended
to a group generated by more than two mutually Engel elements, provided that
none of the generators has order divisible by 2 or 3.

\begin{cor} Let $S$ be a finite Engel set and assume
that $G=\langle S\rangle$ is abelian-by-(nilpotent of class $2$).
If every element in $S$ has order that is not divisible by $2$ or $3$,
then G is nilpotent.
\end{cor}

\begin{proof} For all $x,y\in S$, the subgroup $\langle x,y\rangle$
is nilpotent by Theorem \ref{Ab-by-N2}. Thus the claim follows by Proposition 1
of \cite{ET}.
\end{proof}

Using Theorem \ref{Ab-by-N2}, we now present a criterion for nilpotency of a finite
soluble group depending on information on its Sylow subgroups.

\begin{cor}
Let $G=\langle x,y\rangle$ be a finite soluble group with $x$ and $y$ mutually Engel elements. If all Sylow subgroups of $G$
are nilpotent of class $\le 2$, then $G$ is nilpotent.
\end{cor}

\begin{proof} Let $G$ be a counterexample of least possible order and let $N$ be a minimal normal subgroup of $G$. Then $G/N$ is nilpotent by minimality. Moreover, all Sylow subgroups of $G/N$ are nilpotent of class $\le 2$, so that $G/N$ is nilpotent of class
$\le 2$. On the other hand $N$ is abelian, because $G$ is soluble. Hence $G$ is abelian-by-(nilpotent of class 2) and thus nilpotent
by Theorem~\ref{Ab-by-N2}: a contradiction.
\end{proof}

\section{Examples}

Our first example shows that, for any positive integer $n$, there exists a group generated by two mutually $n$-Engel elements which are not $(n-1)$-Engel.
This is the dihedral group of order $2^{n+1}$.

\begin{exa}
{\rm Let consider $G=\langle x,y\,|\, x^2=y^2=1,(xy)^{2^n}=1\rangle$. If $z=xy$, then
$[x,y]=z^2$ and $z^x=z^y=z^{-1}$. For any $k\geq 1$, we get by induction $[x,\,_k y]=z^{-(-2)^{k}}$ and $[y,\,_k x]=z^{(-2)^{k}}$. Therefore $[x,\,_{n-1} y],[y,\,_{n-1} x]\neq1$ whereas $[x,\,_n y]=[y,\,_n x]=1$. Thus $x$ and $y$ are mutually $n$-Engel elements. Furthermore, we have $G=\langle y,z\rangle$ and $[y,\,_2 z]=[z,\,_n y]=1$, so even $y$ and $z$ are mutually $n$-Engel elements.$\hfill\square$}
\end{exa}

The following is an example obtained by {\sf GAP} of a non-nilpotent group $G$ generated by two mutually 3-Engel elements, for which
$\gamma_4(G)$ is abelian.

\begin{exa}\label{S3}
{\rm Let $W=S_3\,wr\,\mathbb{Z}_4$ be the wreath product of the
symmetric group of degree 3 with the cyclic group of order 4. Thus, $|W|=2^6 3^4$. We
have $W=Q\ltimes N$, where $N$ is an elementary abelian group of order $3^4$ and
$Q\simeq \mathbb{Z}_2\,wr\,\mathbb{Z}_4$. Moreover, $Q$ is
nilpotent of class 4. With the notation of {\sf GAP}, let
ele:=Elements$(W)$, $x:={\rm ele}[4]$ and $y:={\rm ele}[228]$. Then
$o(x)=o(y)=4$ and $[x,\,_3 y]=[y,\,_3 x]=1$. As
$o(xy^{-1})=6$, the subgroup $G=\langle x,y\rangle$ of $W$
is not nilpotent. Finally, one can check that $G=S\ltimes N$ where $S$ is a
group of order $2^5$ which is nilpotent of class 3.

For completeness reasons, we point out that $W=\langle x,y'\rangle$ with $y':={\rm ele}[509]$ of order 6
and $[x,\,_3 y']=[y',\,_4 x]=1$. Hence, $W$ is a generated by two mutually 4-Engel elements and is not nilpotent.
$\hfill\square$}
\end{exa}

Notice that some more non-nilpotent groups generated by two mutually $n$-Engel elements can be found in the literature.
For instance, Corollary 0.2 of \cite{CJ} says that, for $n\geq 26$, the group $G(n)=\langle x,y\,|\,[x,\,_n y]=[y,\,_n x]=1\rangle$
is not nilpotent. We can improve upon this. In fact, we show below that $G(4)$ is not soluble, because it has a quotient isomorphic to the symmetric group~$S_8$.

\begin{exa}\label{S8}{\rm Let $S_8$ be the symmetric group of degree 8,
and let $x=(1,2,3,4)(5,6)(7,8)$ and $y=(1,3)(2,5)(4,7,6,8)$. Put
$x_n=[x,\,_n y]$ and $y_n=[y,\,_n x]$, for any $n\geq 0$ (so $x_0=x, y_0=y$).
We then have:
$$\begin{array}{ll}
x_1=(1,6)(2,7)(3,8)(4,5)\qquad & y_1=(1,6)(2,7)(3,8)(4,5)\\
x_2=(1,5)(4,6)           & y_2=(2,4)(5,7)\\
x_3=(1,5)(2,3)(4,6)(7,8) & y_3=(1,3)(2,4)(5,7)(6,8)\\
x_4=(1)                  & y_4=(1)\,.
\end{array}$$
In particular, $[x,\,_4 y]=[y,\,_4 x]=1$. However $x$ and $y$ are of order 4, but $xy=(1,5,8,6,2)(3,7,4)$ is of order 15.
The subgroup $G=\langle x,y\rangle$ is thus non-nilpotent. Using
{\sf GAP}, it is easy to see that $|G|=8!$, so $G=S_8$.} $\hfill\square$
\end{exa}

We now discuss the situation of Example \ref{S8}. Clearly, if the pair $(x,y)\in G\times G$ satisfies the condition
\begin{equation}
[x,\,_4 y]=[y,\,_4 x]=1,\tag{$*$}
\end{equation}
then all conjugates $(x^g,y^g)$, for all $g\in G$, satisfy the analogous property. Therefore it is  sensible to consider classes under conjugation.

It turns out by {\sf GAP} that the only pairs $(x,y)\in G\times G$ satisfying~$(*)$, that generate a non-nilpotent subgroup of $G$, have both $x$ and $y$ with cycle
structure of type $(4)(2)(2)$ and, in addition, $x,y$ necessarily
generate the whole group $G$. Without loss of generality, we may assume
$x=(1,2,3,4)(5,6)$ $(7,8)$. For this $x$, we calculated
all solutions $y\in G$ of $(*)$. We ended up with precisely
64 solutions. Of course, the group $C_G(x)$ of order 32 acts
on the pairs of solutions. The stabilizer of this action
is $C_G(x)\,\cap\,C_G(y)=Z(G)=1$, so that we obtain two essentially distinct solutions.\\

\noindent{\bf Other examples?} Suppose that in some finite group we can find Sylow $p$-subgroups $P,Q$
and elements $x\in P, y\in Q$ such that $[x,y]\in P\,\cap\,Q$. Let $c$
be the nilpotency class of $P$. Thus, $[x,\,_{c+1} y]=[y,\,_{c+1}x]=1$.
If $xy$ is not a $p$-element, then
$\langle x,y\rangle$ is non-nilpotent. The groups in Examples \ref{S3} and \ref{S8} are of this
form for $p=2$. It would be very interesting to find analogous examples
for all odd primes $p$.

\end{document}